\documentclass[11pt,a4paper, envcountsame]{amsart}
\usepackage[usenames,dvipsnames]{color}

 \usepackage[colorlinks,citecolor=blue,urlcolor=blue, linkcolor=blue, backref]{hyperref}
 
 \usepackage{amsmath, amssymb, xspace}
 \usepackage{graphicx}

\newtheorem{theorem}{Theorem}[section]

\newtheorem{proposition}[theorem]{Proposition}

\newtheorem{prop}[theorem]{Proposition}

\newtheorem{lemma}[theorem]{Lemma}

\newtheorem{corollary}[theorem]{Corollary}

\newtheorem{cor}[theorem]{Corollary}

\theoremstyle{definition}
\newtheorem{definition}[theorem]{Definition}

\newtheorem{remark}[theorem]{Remark}

\newcommand{\NN}{{\mathbb{N}}}

\newcommand{\ZZ}{{\mathbb{Z}}}
\newcommand{\LL}{{\mathbb{L}}}
\newcommand{\FF}{{\mathbb{F}}}
\newcommand{\sub}{\subseteq}
\newcommand{\Fr}{{\rm Frob}}

\renewcommand{\phi}{\varphi}

\newcommand{\bi}{\begin{itemize}}
\newcommand{\ei}{\end{itemize}}
\newcommand{\bc}{\begin{center}}
\newcommand{\ec}{\end{center}}

\newcommand{\ES}{\emptyset}

\newcommand{\ex}{\exists}
\newcommand{\fa}{\forall}

\newcommand{\la}{\langle}
\newcommand{\ra}{\rangle}

\newcommand{\n}{\noindent}

\newcommand{\vsps}{\vspace{3pt}}

\newcommand{\lland}{\, \land \, }

\newcommand \seq[1]{{\left\langle{#1}\right\rangle}}

\newcommand\+[1]{\mathcal{#1}}

\DeclareMathOperator \cost{cost}
\DeclareMathOperator \rank{rk}
\DeclareMathOperator \Hom{Hom}
\newcommand{\wt}{\widetilde}
\newcommand{\ol}{\overline}

\newcommand{\LR}{\Leftrightarrow}

\begin{document}

\title{Describing finite groups by \\ short first-order sentences}

 \author{Andr\'e Nies and Katrin Tent}

\begin{abstract} 
We say that a class of finite structures for a finite first-order signature is $r$-compressible for an unbounded function $r \colon \NN \to \NN^+$ if each structure $G$ in the class has a first-order description of size at most $O(r(|G|))$. We show that the class of finite simple groups is $\log$-compressible, and the class of all finite groups  is $\log^3$-compressible. As a corollary we obtain that the class of all finite transitive permutation groups is  $\log^3$-compressible.
 The   results rely on the classification of finite simple groups, the bi-interpretability of the twisted Ree groups with finite difference fields,   the existence of   profinite presentations with few relators for finite groups, and group cohomology.
 We also indicate why the results are  close to optimal. 
\end{abstract}

\maketitle


\section{Introduction} 
 Let $L$ be first-order logic in a signature  consisting of finitely many relation symbols,  function symbols, and constants.    We say that a sentence $\phi$ in $L$   \emph{describes}    $G$   if $G$ is the unique model of $\phi$ up to isomorphism. 
 We study the compressibility of finite $L$-structures $G$   up to isomorphism via such   descriptions.   Our main results are about compressibility of  finite groups.
 
      Note that  every finite  $L$-structure $G$   can be described by some   sentence~$\phi$:    for each element of $G$ we introduce an existentially quantified variable; we say that these are all the  elements of $G$, and that they satisfy the   atomic formulas valid for the corresponding elements of $G$.   
However,  this  sentence  is at least as long as the  size of the domain of $G$. 
We may think of a  description of $G$ which is much shorter than $|G|$ 
as a \emph{compression}  of $G$ up to isomorphism.   

Unless  stated  otherwise, descriptions of structures  will be in first-order logic.  For an infinite class of $L$-structures, we are interested in giving descriptions that are asymptotically  short relative to the size of the described structure.   This is embodied in  the following   definition. Usually  the function  $r$ grows slowly. 
  \begin{definition} \label{def:compress} Let  $r \colon \NN \to \NN^+$ be  an unbounded      function.  We say  that an infinite  class $\+ C$ of finite $L$-structures   is \emph{$r$-compressible} if for each      structure $G$  in $\+ C$,  there is   a  sentence $\phi$  in $L$  such that  $|\phi| = O(r (|G|))$ and $\phi$ describes~$G$.   \end{definition}
Sometimes  we also want to give a short description of a structure in $\+ C$,  together with a  tuple of elements. We say that    the class $\+ C$   is \emph{strongly  $r$-compressible} if for each      structure $G$  in $\+ C$, each $k$ and each  $\ol g \in G^k$,  there is   a  formula  $\phi(y_1,\ldots, y_k)$  in $L$  such that  $|\phi| = O(r (|G|))$ and $\phi$ describes~$(G, \ol g )$ (where the $O$ constant can  depend on $k$).    

In this paper, for notational convenience  we will use the definition \[\log m = \min \{r \colon \, 2^r \ge m\} .\]
The following is   our first  main result. 
  \begin{theorem}\label{t:simple}  The class of finite simple groups is $\log$-compressible. \end{theorem}  
   Finite  groups can be  described up to isomorphism  via presentations.  There is a large amount of  literature on finding very short presentations for ``most''  finite groups $G$;  see e.g. \cite{Babai:97,Guralnick:08, Bray.etal:11}. Using composition series, these presentations can be converted into first-order descriptions of $G$ that are    at most $O(\log^2 |G|)$ longer, as we will see in Proposition~\ref{p:short_presentation}.   
   
   The small Ree groups $^2G_2(q)$ are finite simple groups that arise as subgroups of the automorphism group $G_2(q)$ of the octonion algebra over the $q$-element  field $\mathbb F_q$, where $q$ has the form $3^{2k+1}$ \cite[Section 4.5]{Wilson:09}.  They form a  notorious case where short presentations are not known to exist. Nonetheless, we are able to find  short first-order descriptions by using the  bi-interpretability with the difference field $(\mathbb F_q, \sigma)$, where $\sigma$ is the $3^{k+1}$-th power of the Frobenius automorphism. This was  proved by 
Ryten~\cite[Prop.\ 5.4.6(iii)]{Ryten:07}.  It then suffices to give a short description of the difference field, which is not hard to obtain.

Let $\log^k$ denote the function $g(n) = (\log(n))^k$. Our second  main result  is the following:

\begin{theorem}\label{t:main} The class of finite groups is  strongly  $\log^3$--compressible.
\end{theorem}

 We describe a general finite group $G$ by choosing  a composition series $1=G_0\lhd G_1\lhd\ldots \lhd G_r=G$,  where $r \le \log |G|$. We use  Theorem~\ref{t:simple} to describe the factors $H_i= G_{i+1}/G_i$ of the series, which are simple by definition. We then use   the method of straight line programs due to \cite{Babai.Szemeredi:84}, and some group extension theory, to obtain short formulas   describing    $G_{i+1}$ for each $i< r$ as an extension of  $G_i$ by $H_i$.   
 
 The proof of    Theorem~\ref{t:main} assuming Thm.~\ref{t:simple} is analogous to the proof of similar results for presentations,   such as Babai et al.\ \cite[Section 8]{Babai:97} and  Mann \cite[Thm.\ 2]{Mann:98}.  However, in our case the deduction is different because we have to describe the extension of $G_i$ by $H_i$ in first-order logic rather than presentations. This is where we use
 the existence of profinite presentations with few relators for finite groups, and group cohomology (Section~\ref{s:ext}).

 Recall that a permutation group is a group $G$ together with an action of $G$ on a set $X$ given by
a homomorphism of $G$ into the symmetric group of~$X$. If 
the action is transitive, then it is equivalent to
the action of $G$ on $H\backslash G$ by right translation, where
$H$ is the stabilizer in $G$ of a point $x\in X$. Thus,
describing the action of $G$ on $X$ amounts to describing $G$ together with a distinguished subgroup $H$ of $G$. Since our methods yield short descriptions of this kind, we obtain:

\begin{corollary}\label{c:permutation} The class of finite transitive permutation groups is   $\log^3$--compressible.
\end{corollary}

 By counting the number of non-isomorphic groups of a certain size, in Remark~\ref{rmk:Ottimal} we  will  also provide  lower bounds on the length of a description, which   show the near-optimality of the two main results. 
 In particular,  from the  point  of view of the length of first-order descriptions,  simple groups are indeed simpler than general finite groups. 
 The lower bounds apply to descriptions in any formal language, such as second-order logic. Thus, for describing finite groups, first-order logic is already optimal.  
 
 As usual in the theory of Kolmogorov complexity, we gauge how short  a description of an object is by comparing it to size of the object itself, considered as its own  trivial description. Given  a   group of size $m$, each element can be encoded by $\log m$ bits. The size of the   table for the operation $(a,b) \mapsto ab^{-1}$ is therefore $m^2 \log m$. This table can be seen as a trivial description. Since our bounds on the lengths of short descriptions are powers of $\log m$, up to a linear constant it does not  matter whether we take $m$ or $m^2 \log m$ as the length of the trivial description.

A \emph{$\Sigma_r$-sentence} of $L$ is a sentence that is  in prenex normal form, starts with an existential  quantifier, and  has  $r-1$ quantifier alternations. We say that $\+C$ is \emph{$g$-compressible using $\Sigma_r$-sentences} if $\phi$ in Definition~\ref{def:compress} can be chosen in $\Sigma_r$ form.  We will   provide variants of the results above where the sentences are  $\Sigma_r$    for    a certain  $r$.  The describing sentences will  be  of length  $O(\log^4|G|)$.

  Usually we view a formula $\phi$ of $L$ as   a string   over the  infinite alphabet consisting of: a finite list of logical symbols, an infinite list of  variables, and the  finitely many symbols of~$L$.  Sometimes we want the      alphabet to be finite, which we can achieve  by  indexing  the variables with numbers written in decimal (such as  $x_{901}$).   This  increases the length of a formula by a logarithmic factor (assuming that $\phi$ always introduces new  variables with the least index that is available, so that $x_i$ occurs in $\phi$ only when $i < |\phi|$).    
 We then encode the resulting string by a binary string, which we call the {\em binary code} for $\phi$.  Its length is called  the {\em binary length} of $\phi$, which is   $O( |\phi| \log |\phi|)$.

Our results are particular to the case of groups. For instance, in the case of all undirected graphs, not much compression is possible using any formal language: the  length of the ``brute force'' descriptions given above, involving the open diagram,   is    close to optimal.  To see this, 
note that there are $2^{\binom n 2}$ undirected graphs on $n$ vertices. The isomorphism class of each such graph has at most $n!$ elements. Hence  the number of non-isomorphic undirected graphs with $n$ vertices is   at least  $2^{\binom n 2}/n! = \frac 1 n \Pi_{i=1}^{n-1} 2^i/i$, which for large $n$ exceeds $\frac 1 n 2^{{n^2}/6}$.    
For each $k$ there are at most $2^k$ sentences $\phi$ with a  binary code  of length less than~$k$.
 So    for  each large enough  $n$ there is an undirected  graph $G$ with $n$ vertices such that $n^2 - 6 \log n = O (|\phi| \log |\phi|)$ for any description $\phi$ of $G$.    (See \cite[Cor. 2.12]{Buhrman.Li.etal:99} for a recent proof that the lower bound $2^{\binom n 2}/n! $ is asymptotically equal to the number of nonisomorphic graphs on $n$ vertices.)

%
%
%
%

\section{Short first-order formulas related to generation}
This section provides short formulas related to generation in monoids and groups. They will be used later on to obtain descriptions of finite groups. 
Some of the results are  joint work with Yuki Maehara, a former project student of Nies.

Firstly, we consider exponentiation in monoids. 
\begin{lemma} \label{repeated squaring} 
For each positive integer $n$, there is an     existential  formula $\theta_n(g,x)$   in the first-order  language of monoids  $L(e, \circ)$, of length $O(\log  n)$, such that for  each   monoid~$M$, 
$M \models \theta_n(g,x)$ if and only if $x^n=g$.  

\end{lemma}

\begin{proof} We use a standard method from the theory of algorithms  known as   {exponentiation via repeated squaring}. 
Let $k=  \log  n$.
 Let $\alpha_1\ldots\alpha_k$   be the  binary expansion of $n$. Let  $\theta_n(g,x)$ be the formula
\[
  \exists y_1 \ldots \exists y_k  [y_1=x \ \wedge \ y_k=g \ \wedge \ \bigwedge_{1 \le i < k} y_{i+1}=y_i \circ y_i \circ x^{\alpha_{i+1}} ] \tag{*}
\]
where $x^{\alpha_i}$ is $x$ if ${\alpha_i}=1$,  and $x^{\alpha_i}$ is $e$ if ${\alpha_i}=0$. Clearly $\theta_n$ has length $O(\log\ n)$.
One  verifies  by induction on $k$  that the formulas  are correct.
%
\end{proof}

We give a sample application of Lemma~\ref{repeated squaring}  which will also be useful below.    By  the remark after Prop.\ \ref{rem:fields optimal} below,  the upper bound on the length of the descriptions   is close to optimal.
\begin{proposition}\label{c:cyclic}
The class of    cyclic groups $G$    of prime power order is   $\log$-compressible via sentences in the language of monoids that are in $\Sigma_3$ form. 
\end{proposition}
\begin{proof} Suppose that  $n= |G|= p^k$ where $p$ is prime. A group $H$ is isomorphic to $G$ if and only if there is an element $h$  such that $h^{p^k} = 1$, $h^{p^{k-1}}\neq 1$, and~$h$ generates $H$. By Lemma  \ref{repeated squaring}, the first two conditions  can be expressed by formulas of length $O(\log n)$ with $h$ as a free variable, the first  existential, the second universal. For the third condition, we need a   modification of  Lemma  \ref{repeated squaring}, namely a formula $\chi_n$ such that  for  each   monoid~$M$, 
$M \models \chi_n(g,x)$ if and only if $x^r=g$ for some $r $ such that $0 \le r < 2^{k}$, where $k = \log n$ (recall that by our definition of $\log$, the number $2^k$ is the least power of two which is not less than $n$). We  define $\chi_n(g,x)$   by 
\begin{equation*} 
  \exists y_0 \ldots \exists y_k  [y_0=1   \ \wedge \ y_k=g \ \wedge \bigwedge_{0 \le i < k} (y_{i+1}=y_i \circ y_i \circ x  \vee y_{i+1}=y_i \circ y_i) ] \
\end{equation*}
 It is now clear that the condition that $h$ generates the group  can be expressed by a formula of length $O(\log n)$ which is in $\Pi_2$ form.
\end{proof}

For elements $x_1,\ldots x_n$ in a group $G$ we let $\seq { x_1,\ldots,x_n }$ denote the subgroup of $G$ generated by these elements.
The pigeon hole principle easily implies the following:
\begin{lemma} \label{finite product}
Given a generating set $S$ of a finite group $G$, every element of $G$ can be written as a product   of elements of $S$ of length at most $|G|$. 
\end{lemma}

We     define a crucial collection of   formulas  $\alpha_k(g;  x_1,\ldots,x_k)$  in the first-order    language of monoids    so that $\alpha_k(g;  h_1,\ldots,h_k)$ expresses  that
$g$ is in    ${\langle h_1,\ldots,h_k \rangle}$.    These formulas depend only on $k$ and the size of the group~$G$.
\begin{lemma} \label{generation}
For each positive integers $k,v$, there exists a first-order formula $\alpha_k(g;x_1,\ldots,x_k)$   in the language of monoids of length $O(k+\log v)$ such that   for each group $G$ of size at most $v$,  ${G \models \alpha_k(g;x_1,\ldots,x_k)}$ if and only if ${g \in \langle x_1,\ldots,x_k \rangle}$. 
\end{lemma}
\begin{proof}
We use a technique that originated in computational complexity to show that the set of true quantified boolean formulas is PSPACE-complete. 
For  $i \in \NN$ we inductively define   formulas $\delta_i(g;x_1,\ldots,x_k)$.
Let
\[\delta_0(g;x_1,\ldots,x_k) \equiv \bigvee_{1 \le j \le k}[g=x_j \  \vee \ g=1].\] 
For $i>0$ let
\[
\begin{split}
\delta_i(g;x_1,\ldots,x_k) \equiv \exists u_i \exists v_i [&g = u_i v_i \ \wedge\\
&\forall w_i [(w_i = u_i \vee w_i = v_i) \rightarrow \delta_{i-1}(w_i;x_1,\ldots,x_k)]].
\end{split}
\]
Note that $\delta_i$ has length $O(k+i)$, and $G \models \delta_i(g;x_1,\ldots,x_k)$ if and only if $g$ can be written as a product, of length at   most $2^i$,   of $x_r$'s.

Now let  $\alpha_k(g;x_1,\ldots,x_k) \equiv \delta_p(g;x_1,\ldots,x_k)$  where $p =  \log v $. Then $2^p \ge v $ by our  definition of $\log$, so $\alpha_k$ is a formula as   required by Lemma~\ref{finite product}.
\end{proof}

\begin{remark} \label{binary} We note that we can optimize the formulas     in Lemmas~\ref{repeated squaring} and Lemma~\ref{generation}  so that  the length bounds apply to the binary length.  For instance, in Lemma~\ref{generation}   we can ``reuse''  the quantified variables $u,v,w$ at each level $i$, so that $\alpha_k$ becomes a formula  over an alphabet of size $k + O(1)$. \end{remark}

%

\section{Straight line programs and  generation} \label{s:SLP}
In this section we recall the Reachability Lemma from  Babai and Szemer\'edi~\cite[Thm.\ 3.1]{Babai.Szemeredi:84}, and the notion  of a pre-processing set introduced in  Babai et al.~\cite[Lemma 8.2]{Babai:97} following their proof sketch. Let $G$ be a finite group, $S\sub G$
and $g\in G$.
 A~\emph{straight line program (SLP) $\+ L$  over $S$} is a sequence of group elements such that each element of $\+ L$  is either in $S$, an inverse of an earlier element or a product of two earlier elements. We say that an SLP $\+L$
\emph{computes $g$ from $S$} if $\+L$ is an SLP over $S$ containing $g$.

 The \emph{reduced length} of $\+ L$ is the number of elements in $\+ L$  outside $S$.  For a set $A\sub G$ we say that a straight line program $\+ L$ over $S$ \emph{computes $A$} if  every element of $A$ occurs in $\+L$.
  Let $\cost(A\mid S)$ be the shortest reduced length of a straight line  program computing $A$ from $S$.
 

For a  subset $S$ of a finite group $G$, Babai  and Szemer\'edi~\cite{Babai.Szemeredi:84} construct a set of generators $A$ for $\seq{S}$ with  $|A|\leq \log|\seq{S}|$ such that every element of $\seq{S}=\seq{A}$ has length at most $2\log|G|$ as a word over $A$ (cf. Lemma \ref{finite product}).
Such \emph{pre-processing sets} will reduce  the length of the formulas in Section~\ref{s: general}.  
We include the construction for convenience,   and in order to adjust it for future reference to    an increasing sequence of subsets of $G$.

 \begin{lemma}[\cite{Babai.Szemeredi:84,Babai:97}] \label{lem:Babai_super} Let $G$ be a finite group.   Suppose  $T_1 \subset\ldots\subset T_k\sub G$ is an ascending sequence of subsets
 and $G_i=\seq{T_i}, i=1,\ldots, k$. 
 
 There is an ascending sequence of pre-processing sets   $A_i$ for $G_i,i=1,\ldots k,$ with $|A_i|\leq \log|G_i|, \seq{A_i}=\seq{T_i}$, $\cost (A_i \mid T_i)< (\log|G_i|)^2$,    and $ \cost(g \mid A_i) < 2 \log |G_i|$ for every  $g \in G_i$.
  \end{lemma}
  \begin{proof} 
We  first consider the case $k=1$, i.e., a single set $S=T_1\sub G$. Let $s$ be minimal with $2^s\geq |\seq{S}|$ so that $s = \log |\seq{S}|$ according to our definition of $\log$.  For $i \le s$, we inductively define an increasing sequence of  subsets $K(i)\sub \langle S\rangle$ of size $2^i$, elements $z_i\in K(i)$  and an increasing sequence of SLPs $\+ L_i$ computing  $z_1,\ldots, z_i$ from $S$.
  The set $\{z_1,\ldots, z_s\}$ will   serve  as our pre-processing set for $\seq{S}$.

To begin with, let  $K(0) = \{1\}, z_0=1, \+ L_{0}=\emptyset$. Suppose $K(i)$ and $\+ L_{i}$ have been defined with the required properties. If $  K(i)^{-1} K(i) \neq \seq S$,    there are $v \in K(i)^{-1} K(i)  $ and $x \in S$ such that $z_{i+1} := vx \not \in K(i)^{-1} K(i) $. 
Let $K(i+1) $ be the set of products $\Pi_{l \le i+1} z_l^{\alpha_l}$ where $\alpha_l \in \{ 0,1\}$.
By the choice of $z_{i+1}$ we have $|K(i+1)| = 2 |K(i)|=2^{i+1}$.

One   can write $z_{i+1} = v_0^{-1} v_1 x$ with $v_0,v_1 \in K(i)$, and $v_0 \neq v_1$,  so that we can also assume that not both $v_0$ and $v_1$ have length $i$.  
Since $\+L_i$ computes $z_1,\ldots z_i$, an SLP  computes  $v_0$ and $v_1$ which    extends $\+L_{i}$  by  at most $2i-3$   elements (corresponding to the initial segments of the $v_r$ of length $>1$). 
 To obtain $\+ L_{i+1}$,  we append  $v_0^{-1}$, $v_0^{-1}v_1$ and finally $z_{i+1}=v_0^{-1}v_1x$ to  $\+ L_{i}$. In total we have appended at most $2i$ elements to $\+ L_i$.
Clearly this process ends after $s$ steps, when $  K(s)^{-1} K(s) = \seq S$.
Then $A=\{z_1,\ldots z_s\}$ is a generating set for $\seq{S}$, and $\+L_s$
is an SLP computing $A$ from $S$ of reduced length at most $2 \sum_{i=1}^s  (i-1)   \leq s^2$. Since $  K(s)^{-1} K(s) = \seq S$, we see that any $g\in \seq {S}$ can be computed from $A=\{z_1,\ldots z_s\}$
by an SLP of reduced length at most $2s-1  <   2\log |G|$. Thus, $A$ is the required pre-processing set.

Now suppose we have $S= T_1\subset T_2$ and   $A = A_1$ as above  computed by 
$ \+ L_{s_1}=\+ L_s$ from $T_1$ so that $K(s_1)^{-1} K(s_1) = \seq {T_1}$. We continue the construction using elements   $x \in T_2$ for   $s_2 -s_1$ steps extending $A_1$ to a set $A_2$ and $ \+ L_{s_1}=\+ L_s$  to an SLP $\+ L_{s_2} $. Inductively we find the required $A_i, i\leq k$.
\end{proof}

\begin{corollary}\label{c:swift_generation}Any finite group $G$ has a  generating set $A$ of size at most
$\log|G|$ such that any element of $G$ has length at most $2|A|$ over $A$.
\end{corollary}

 We call   a generating set with the latter property  \emph{swift}. Swift generating sets  will be used below to give short descriptions for finite groups.

For reference we also note that the first part of the proof of Lemma \ref{lem:Babai_super} shows the following:

\begin{corollary}[Reachability Lemma \cite{Babai.Szemeredi:84}]\label{lem:SLP} Let $r = \log |G|$. For each   set $S \sub G$ and any $g\in\langle S\rangle$,   
there is a  straight line program   $\+ L$ of reduced length  at most $(r+1)^2$ that  computes $g$ from $S$. \end{corollary} 
\begin{proof}
We build the sequences $z_1,\ldots, z_i$, $K(0),\ldots, K(i)$ and $ \+L_0,\ldots, \+L_i$ as in the proof of Lemma \ref{lem:Babai_super}
until $g\in K(s)^{-1} K(s)$. This yields an SLP computing $g$ from $S$ of reduced length at most $(r+1)^2$.
\end{proof}

We say that an $L$-formula $\phi(\ol x)$ in variables $\ol x$, possibly with parameters in  a group $G$,   \emph{defines} an ordered tuple $A$ (of the same length as $\ol x$) in $G$ if $A$ is the unique
tuple in $G$ such that $\phi(A)$ holds in $G$ (here    the elements of $A$ are substituted for the variables of $\phi$).  We can define swift generating sets for a
normal series of a group $G$ by a formula in $O(\log^2|G|)$.

\begin{lemma}\label{l:preproc_formula}
Let $G$ be a finite group with a normal series
\[1\lhd G_0\lhd G_1\lhd\ldots \lhd G_r=G\]
and an ascending sequence of generating sets 
\[ T_0\subset T_1\subset\ldots\subset T_r=T\]
with $\seq{T_i}=G_i, |T_i|\leq\log|G_i|, 0 \le i\leq r$.

There is a formula $\psi$ using  parameters from the set $T$   
with $|\psi|= O(\log^2|G|)$
defining a  sequence $A_1 \sub  \ldots \sub A_r =A\sub G$  of pre-processing sets for $G_i$ over $T_i, i\leq r$.

\end{lemma} 
\begin{proof}
Note first that $r\leq \log|G|$.
Let  $A_1 \sub  \ldots \sub A_r =A\sub G$  be pre-processing sets computed by SLP's $\+ L_1 \prec \ldots \prec \+ L_r=\+ L$ from $T_1 \subset  \ldots \subset  T_r=T$  according to Lemma~\ref{lem:Babai_super}. Recall that the reduced length of $\+ L_i$ is at most $\log^2 |\seq{T_i}|$.

The formula  $\psi$   in free variables corresponding to the elements of  $T$ and~$A$ expresses that $A$ is a preprocessing set for $G$ over $T$. We first  build a formula $\psi_0$ in the same free variables. We  start with a prenex of  existential quantifiers that refers  to the
sequence of elements of the SLP $\+ L$ that are not in $A$. The formula  $\psi_0$ expresses    each member of the sequence as the product of two previous elements, or inverse of a previous element,  according to $\+L$.  Then $\psi_0$ has  length   $O(\log^2 |G|)$.

To build $\psi$, we use the formulas $\alpha_{|T_i|}(y, T_i)$ and $\alpha_{|A_i|}(y,A_i)$ from Lem\-ma~\ref{generation},  of length $O(\log |G_i|)$, to also  express that $\seq{A_i}=\seq{T_i}, i\leq r$. Then  $\psi$ has length $O(\log^2|G|)$. \end{proof}  

%
The formulas $\alpha_k$   in Lemma~\ref{generation}   have  about $2\log v$  quantifier alternations for $k>0$, and use negation. Via  the argument in Lemma~\ref{lem:Babai_super}, we can obtain existential formulas  without negation symbols  that are somewhat longer.   
 
\begin{lemma} \label{generation2}
For each  pair of positive integers $k,v$, there exists an existential   negation-free
  first-order formula  $\beta(g;x_1,\ldots,x_k)$   of length $O(k \log v +\log^2 v)$ such that   for each group $G$ of size at most $v$,  \bc ${G \models \beta(g;x_1,\ldots,x_k)}$ if and only if ${g \in \langle x_1,\ldots,x_k \rangle}$. \ec
\end{lemma}
We note that in applications we will have $k \le \log v$ so that the length is $O(\log^2 v)$.
\begin{proof}     The formula describes the generation of a SLP computing  $g$ from $x_1, \ldots, x_k$ according to the  proof of Lemma~\ref{lem:Babai_super},  for  a single set $S$.
The existentially quantified variables $z_1, \ldots, z_s$, where $s= \log  v$,  correspond to  the preprocessing set $A$, while $z_{s+1}$ equals $g$. The rest of the formula expresses that each $z_{t+1}$ for $t<s$ has the form $v_0^{-1}v_1x$ for $v_0, v_1 \in K(t)$ and a generator $x$, and that $g = v_0^{-1}v_1$ for $v_0, v_1 \in K(s)$. In detail,  let $ \beta(g;x_1,\ldots,x_k) \equiv$
\[ 
\begin{split}\ex z_1, \ldots , z_{s+1} \ex y_0, \ldots , y_s  [ g= z_{s+1} \land y_s = 1 \land (\bigwedge_{0\le i < s} \bigvee_{1 \le r \le k} y_i = x_r ) \land  \\
\bigwedge_{0 \le t \le s} \ex p_0, \ldots, p_t \ex q_0, \ldots, q_t [p_0=q_0 =1 \land z_{t+1} = p_t^{-1} q_t y_t  \ \land \ \ \ \ \ \ \ \  \\
 \bigwedge_{0 \le j <  t}  [(p_{j+1} = p_j z_j \lor p_{j+1}= p_j) \land (q_{j+1} = q_j z_j \lor q_{j+1}= q_j)]]] \end{split} \]

 \n  Clearly $\beta$ has length $O(k \log v +\log^2 v)$.
\end{proof}



\section{Describing finite fields and finite difference fields}

 Recall that a finite field $\FF$ has size  $q=p^n$ where $p$ is a prime called the characteristic of~$\FF$. For each such $q$  there is a unique field $\mathbb F_q$ of size $q$. Let  $\Fr_p$ denote the Frobenius automorphism $x \to x^p$  of  $\mathbb F_q$.  The group of automorphisms of $\mathbb F_q$ is cyclic of order $n$ with $\Fr_p$ as a generator. In particular,  $(\Fr_p)^n$ is the identity on $\FF_q$.
 
A  \emph{difference field} $(\FF,\sigma)$ is a field $\FF$ together with
a distinguished automorphism~$\sigma$. Examples are the field of complex numbers with complex conjugation
and  finite fields of characteristic $p$ with  a fixed power of the Frobenius automorphism.  We show that finite   fields and finite difference fields 
are    $\log$-compressible in  the language of rings $L(+, \times, 0,1)$, extended by a unary function symbol $\sigma$ in the second case. Besides providing another example for our main Definition~\ref{def:compress}, this will be used in one case of the proof of our first  main result,  Theorem~\ref{t:simple}.

  \begin{proposition}\label{L:fields}   \mbox{}  
  
\n   {\rm (i)}   For any  finite field $\mathbb F_q$,   there is a  $\Sigma_3$-sentence  $\varphi_{q} $ of length $O (\log q)$  in  $L(+, \times, 0,1)$   describing~$\mathbb F_q$.

\n {\rm (ii)}  For any finite difference field  $(\mathbb F_q, \sigma)$ there is a $\Sigma_3$-sentence  $\psi_{q, \sigma} $  of length $O (\log q)$   in $ L(+, \times, 0,1, \sigma)$    describing $\la \mathbb F_q, \sigma\ra $. 

\n   {\rm (iii)}   For any  finite field $\mathbb F_q$, and  any   $c\in \mathbb F_q$, there is  a  $\Sigma_3$-formula   $\varphi_{c}(x) $ of length $O (\log q)$ in  $L(+, \times, 0,1)$   describing the structure $\la \mathbb F_q, c \ra$.
\end{proposition}

\begin{proof} {(i).}  The sentence  $\phi_q$ says that the structure is a field of characteristic $p$ such that for all elements $x$ we have $x^{p^n}=x$ and there
is some $x$ with  $x^{p^{n-1}}\neq x$. By Lemma~\ref{repeated squaring}  one can ensure that $|\phi_q| = O (\log q)$ and the sentence $\phi_q$ is $\Sigma_3$.

\n {(ii).} Since any automorphism of $\FF_q$ is of the form $(\Fr_p)^k$
for some $k\leq n$, we can use  Lemma~\ref{repeated squaring}   again   to
find a sentence of length $O(\log q )$ expressing that
$\sigma(x)=x^{p^k}$ for each~$x$.

\n {(iii).}  
  By (i) it suffices to give a formula that determines   $c$ within  $\FF_q$ up to an automorphism of the field. Let $q=p^n$ as above.   Since $\FF_q$ is a Galois extension of $\FF_p$ of degree $n$, any  $c\in\FF_q$  is determined within $\FF_q$ up to an automorphism of the field  by being a zero of its  minimal
polynomial over $\FF_p$. This polynomial  has degree at most $n$. So we can apply  Lemma~\ref{repeated squaring} repeatedly to express that $c$ is a zero of the polynomial  by a   formula  of  length $O(n \cdot \log p)= O(\log q)$.   
\end{proof} 

      
%
%
%
%
%
 
\begin{corollary}\label{c:fieldtuple}
The class of   finite fields is strongly $\log$-compressible (as defined after Def.\ \ref{def:compress})  via   $\Sigma_3$-sentences  in  $L(+, \times, 0,1)$. 
\end{corollary}

\begin{proof}
By Proposition~\ref{L:fields}(iii), a generator $b$ of the
multiplicative group of $\FF_q$ can be determined within $\FF_q$  up to automorphism by a $\Sigma_3$-formula  $\varphi_{q}(x) $  of length $O (\log q)$. In order to determine  a finite tuple of field elements  up to automorphism,   it thus suffices to pin down  the corresponding tuple of exponents of $b$. Since these exponents are
bounded by $q-1$, via Lemma~\ref{repeated squaring} this can be done  with  a formula of length $O (\log q)$.
\end{proof}
The following shows   that the  upper bound of $O(\log q)$ on the length of a sentence  describing $\mathbb F_q$  is  close to  optimal for infinitely many $q$.   
\begin{proposition}  \label{rem:fields optimal}  There is a constant $k>0$ such that for infinitely many primes $q$, for any      description $\varphi$ for $\FF_q$, we have 
\bc $ \log (q )  \le k |\varphi| \log  |\varphi|$.\ec
\end{proposition} 
 
  \begin{proof} Let $C(n)$ denote  the Kolmogorov complexity of the binary expansion of a natural number~$n$.   A sentence  $\varphi$  describing $\mathbb F_q$ also yields a description of  the number~$q$. Therefore  $C(q) \le  k'  |\varphi| \log  |\varphi| $ for some $k'$, where the corrective     factors  are  needed because the  string $\varphi$ over an infinite alphabet has to be encoded  by a binary string in order to serve as a description in the sense of Kolmogorov complexity.
  
  Infinitely many  $n\in \NN$ are  random numbers,  in that  $C(n) =^+ \log_2 n$  (the superscript $+$ means that the inequality holds up to a constant). Now let $q= p_n$, the $n$-th prime number, so that $C(q) =^+ C(n) = ^+ \log_2 n$. By the prime number theorem  $p_n/ \ln (p_n)  \le 2n$ for large $n$,  so that  $   \log (q/\ln q) \le^+ \log_2 n$. Note that  $\sqrt q \le q/\ln q $ for $q \ge 3$ so that $\log q -1 \le \log (q/\ln q)$. Choosing $k \ge k'$ appropriately and putting the inequalities together, we obtain  $\log q  \le  k  |\varphi| \log  |\varphi| $ as required.  \end{proof}

A similar  argument  shows that Proposition~\ref{c:cyclic}, for  descriptions of cyclic groups of prime order, is close to optimal.

\section{Describing finite simple groups}
 
The main result of  this section is the following.

 \noindent
{\bf Theorem \ref{t:simple}.}\emph{   The class of finite simple groups is  $\log$-compressible.  }

\n We do not know whether the class of finite simple groups is \emph{strongly} $\log$-compressible (cf. Lemma \ref{lem:short_pres}, but see also Propositions~\ref{p:Ree}).
For the proof of Theorem~\ref{t:simple}, recall that any finite simple group belongs to one of the following classes:


\begin{enumerate}
\item the   finite cyclic groups $C_p,p$ a prime;
\item the     alternating groups $A_n, n\geq 5$;
\item the   finite simple groups $L_n(\FF_q)$ of fixed Lie type $L$ and Lie rank $n$, possibly twisted, over a finite field $\FF_q$;
\item the  $26$ sporadic simple groups.
\end{enumerate}
See e.g.\ \cite{Wilson:09}, Section 1.2.
\subsection{Short first-order descriptions via short presentations}  \label{ss:shsh}
Clearly for the proof of Theorem \ref{t:simple} we may disregard the finite set of sporadic simple groups. 
For most of
the other classes we will use that there exist short presentations. 
Recall   that a finite presentation of a group $G$ is given by a normal subgroup $N$ of a  free group $F(x_1, \ldots, x_k)$ such that $G =  F(x_1, \ldots, x_k)/N$, and $N$ is generated as a normal subgroup by relators $r_1, \ldots, r_m$. One  writes  $G= \la x_1, \ldots, x_k \mid r_1, \ldots , r_m \ra$.
\begin{definition}  
We define the \emph{length} of a presentation 
  $G= \la x_1, \ldots, x_k \mid r_1, \ldots , r_m \ra $
to be  $k +  \sum_j |r_j|$, 
where $|r_j|$ denotes the length of the relator $r_j$ expressed
as a word in the generators $x_i$ and their inverses.
\end{definition}

  \begin{lemma} \label{lem:short_pres} Suppose that a finite simple group $G$ has a presentation 
  
  \n
  $ \la x_1, \ldots, x_k \mid r_1, \ldots , r_m \ra$ of length $\ell$. Let $g_i $ be the image of $x_i$ in $G$, $i=1,\ldots, k$.
  
\n (i) There is a sentence  $\psi$  of length $ O(\log |G| + \ell)$ describing the structure  $\la G, \ol g \ra$.  

\n
(ii) There is a   $\Sigma_3$-sentence  $\psi$  of length
   $O( \log^2|G|  +    \ell)$  describing the structure  $\la G, \ol g \ra$, provided that $k \le \log |G|$.  
       \end{lemma}
  \begin{proof} (i).   Let $\psi$  be 
\bc   $ x_1 \neq 1 \lland \bigwedge_{1 \le i \le m}  r_i = 1 \lland    \fa y \, \alpha_k (y; x_1, \ldots, x_k),$ \ec
  where $\alpha_k$ is the formula from Lemma~\ref{generation}  of length   $O(k+\log|G|)$   expressing  that $y$ is   generated from the $x_i$ within $G$. Replacing the $x_1,\ldots, x_k$ by new constant symbols, the models of the sentence thus obtained are the  nontrivial quotients of $G$. Since $G$ is simple, this sentence describes~$\la G, \ol g \ra$.
  
  \n 
   (ii) is similar, using the formula $\beta_k$ from Lemma~\ref{generation2} instead of $\alpha_k$.
\end{proof}

For most classes of finite simple groups, Guralnick et al.\ \cite{Guralnick:08} obtained a  presentation for each member $G$ that is very short compared to $|G|$. 

\begin{theorem}{\rm \cite[Thm.\ A]{Guralnick:08}}  \label{thm:Guri} There is a constant $C_0$ such that 
any nonabelian finite simple group, with the possible exception of the 
Ree groups  of type $^2G_2$, has a presentation with at most $C_0$ generators and 
relations and   length at most $C_0(\log n +\log q)$, where $n$ denotes the Lie  rank of the group and $q$
the order of the corresponding field.
\end{theorem}
Note that,  following Tits, they  considered the alternating groups $A_n$
as groups of Lie rank $n-1$ over the ``field''  $\FF_1$ with one element.  For more detail see their  remark before \cite[Thm~A]{Guralnick:08}.

\begin{proposition}\label{p:simple}  (i).  The class of finite simple groups, excluding  the   Ree groups of type $^2G_2$, is $\log$-compressible. 

\n (ii). The same class is  $\log^2$-compressible using $\Sigma_3$-sentences. 
\end{proposition} 
\begin{proof} For  cyclic simple    groups, this   follows from  Proposition~\ref{c:cyclic}. Now consider a finite simple group $G=L_n(\FF_q)$, that is, $G$ is  of Lie rank $n$ with corresponding field $\FF_q$. Suppose $G$  is not a    
   Ree group of type $^2G_2$.  We have
   $\log n+\log q \le \log |G|$: This is clear for the alternating groups $A_n$ because $q=1$ and  $|A_n|=n!/2$. Otherwise,   the calculations of sizes of finite simple groups in e.g.\ 
   \url{http://en.wikipedia.org/wiki/List\_of\_finite\_simple\_groups} (August 2014) or  Wilson~\cite{Wilson:09}  show that $|G| $ is at least $q^n$. 

   Now by   the foregoing theorem,  together with   Lemma \ref{lem:short_pres}~(i) replacing the constants  by  variables $x_i$, we obtain a formula $\psi(x_1, \ldots, x_{C_0})$ of length  $O(\log |G|)$. Then the sentence  $\phi \equiv \ex x_1 \ldots \ex x_{C_0} \, \psi$ is as required for (i). For (ii) we use  Lemma \ref{lem:short_pres}~(ii) instead. \end{proof}

 We also note the following:
 
\begin{proposition}\label{p:short_presentation}
For any function $f:\NN\to \NN$, the class of finite groups $G$
with a presentation of total length $f(|G|)$ is strongly $(f+\log^2)$-compressible.

\end{proposition}
\begin{proof}
Suppose $G$ has a presentation
 $G=\langle x_1,\ldots, x_k\mid r_1,\ldots , r_m\rangle$
of   length $f(|G|)$.
Fix a composition series
\[1\lhd G_1\lhd\ldots \lhd G_r=G\]
and an ascending sequence of swift generating sets  (see Cor.\ \ref{c:swift_generation})
\[ A_0\subset A_1\subset\ldots\subset A_r=A\]
with $\seq{A_i}=G_i, |A_i|\leq\log|G_i|, 0 \le i\leq r$. Note that $r\leq\log|G|$.

We start with a prenex of existential quantifiers referring to the elements
of $A$ and then express  that for each $i$ the subgroup generated by $A_i$ is a proper normal subgroup of the
subgroup generated by $A_{i+1}$, using the $\alpha_k$ from   Lemma~\ref{generation} for  $k= |A_i|$.
This takes length $O(\log^2 |G|)$.

We next express the $x_1,\ldots, x_k$ as words over the
preprocessing set $A$. This takes a length of $2|A|\cdot k$. We note that 
the formula
 \[ \bigwedge_{1 \le i \le m}  r_i = 1 \lland    \fa y \, \alpha_k (y; x_1, \ldots, x_k),\]
holds in a group $(H,\ol h)$ if and only if $(H,\ol h)$ is a quotient of $(G,\ol g)$ where $\ol h,\ol g$ are the images of $\ol x$ in $H$ and $G$, respectively.
Since a composition series of a proper quotient of $G$ is shorter than $r$, we see that the conjunction of these three formulas describes $(G,\ol g)$ with 
a length of $O((f+\log ^2)|G|)$. For strong compressibility, note that  any tuple of elements from $G$ can be written as a word of length $2|A|$ over $A$.
\end{proof} 
 

 It was shown in \cite{Babai:97} that any finite group $G$ without a 
 composition factor of type $^2G_2$ has a presentation of length $O(\log^3|G|)$. Hence we obtain:
   
\begin{corollary}\label{c:solv}
The class of finite solvable groups, and more generally of groups without  a composition factor of type $^2G_2$, is strongly
$\log^3$-compressible.
\end{corollary} 
While this also follows from our main result Thm.\  \ref{t:main} proved below, it is interesting to note that this restricted form can be obtained already as this stage.
\begin{remark}\label{r:strong_simple} Using the  argument  of  Proposition~\ref{p:short_presentation}, one can see
that if a class  of finite groups is $f$-compressible for some function $f:\NN\to\NN$, then this class is strongly $(f+\log^2)$-compressible. However, we do not
know whether the class of finite simple groups is strongly $\log$-compressible.
\end{remark} 

\subsection{Short first-order descriptions via interpretations} \label{ss:interpretations}
It  remains to treat the class of   Ree groups of type $^2G_2$. While the Chevalley groups of type
$G_2$ exist over any field  $\FF$  as the automorphism group of the octonion algebra over $\FF$, the (twisted) groups $^2G_2$
exist only over fields of characteristic $3$ which have an automorphism~$\sigma$ with  square 
the Frobenius automorphism. For  a finite field $\FF_q$,  this happens if and only if $q=3^{2k+1}$. The untwisted group  can be presented as a matrix group over such a field. The twisted group can be seen as the group of fixed points under a certain automorphism of $G_2$ arising from the symmetry in the corresponding Dynkin diagram, which  induces $\sigma$ on the entries of the matrix (see e.g.
\cite[Section 13.4]{Carter:89}). 

 Strong $r$-compressibility was  introduced  after Definition \ref{def:compress}.
\begin{proposition}\label{p:Ree} The class of   Ree groups of type $^2G_2$ is strongly $\log$-compressible via $\Sigma_d$-sentences for some constant $d$.
\end{proposition}

No short presentations are known for these Ree groups. Instead, we use  first-order interpretations between groups and finite difference  fields in order  to derive the proposition from Lemma~\ref{L:fields}.

		 Suppose that $L,K $ are languages in  finite signatures. Interpretations via first-order formulas of $L$-structures in $K$-structures are formally defined, for instance,   in \cite[Section 5.3]{Hodges:93}.   Informally, an $L$-structure $G$ is interpretable in  a $K$-structure $F$ if the elements of $G$ can be represented by tuples in a definable $k$-ary relation $D$ on $F$, in such a way that equality of $G$  becomes an $F$-definable equivalence relation $\approx$ on $D$,  and the other atomic  relations on $F$ are also definable. 
		 
		 A simple example is the   field of fractions of a given intergral domain, which can be interpreted in the  domain.  For an example more relevant to this paper,      fix $n \ge 1$. For any field $\FF$,  the linear group $SL_n(\FF)$  can be  interpreted in $\FF$. A matrix $B$ is represented by a tuple of length $k=n^2$, $D$ is given by the  first-order condition that $\det (B) = 1$, and $\approx$     is equality of tuples. The group operation of $SL_n(\FF)$ is then given by matrix multiplication, and can be 
expressed in a first-order way using  the field operations.

We think of the interpretation of $F$ in $G$
as a decoding function $\Delta$. It decodes $F$ from $G$ using  first-order formulas, so that $F = \Delta(G)$ is  an $L$-structure.   

 

		\begin{definition}\label{def::interpretation}  {\rm  Suppose that $L,K $ are languages in a finite signature, and   that    classes $\+ C \sub M(L), \+ D \sub M(K)$ are given. 	 We say that a function $\Delta$ as above  is a \emph{uniform interpretation   of $\+ C$ in $\+ D$} if for each $G \in \+ C$, there is $F \in \+ D$ such that $G = \Delta(F)$.		} \end{definition} 		
		  Note that if $\Delta$  is a uniform interpretation   of $\+ C$ in $\+ D$, then there is some $k\in\mathbb N$, namely the arity of the relation $D$, such that for $G = \Delta(F)$ we have $|G|\leq |F|^k$. 
		  
		  For example, the class of special linear  groups $SL_2(\FF)$ over
		finite fields $\FF$ is uniformly interpretable in the class of finite fields via the decoding function $\Delta$ given by the formulas 		above.
		
		Suppose $K'$ is  the signature $K$  extended by a finite number of    constant symbols. Let $\+ D'$ be the class of  $K'$-structures, i.e.  $K$-structures giving values to these constant symbols. We say that   a function $\Delta$ based on first-order formulas in $K'$   is a  \emph{uniform interpretation   of $\+ C$ in $\+ D$ with parameters} if $\Delta$     is a uniform interpretation   of $\+ C$ in $\+ D'$.
		
		We will apply 
		 the following proposition to the class   $\+ C $ of finite   Ree  groups of type $^2G_2(q)$, and  the class $\+ D$   of finite difference fields for which these Ree groups exist. 

\begin{proposition}\label{t:biinterpretation}  {\rm Suppose that $L,K $ are languages in a finite signature, and   that    classes $\+ C \sub M(L), \+ D \sub M(K)$ are given. 	 Suppose  furthermore that 

(1) there is a uniform interpretation $\Delta$ without  parameters of $\+ C$ in $\+ D$, 

(2) there is a uniform interpretation $\Gamma$ with parameters  of $\+ D$ in $\+ C$,  
and 

(3) there is an $L$-formula $\eta$  involving   parameters  such that  for each $G \in \+ C$ there is a list of parameters $\ol p$ in $G$ so that $\eta$ defines an isomorphism between $G$ and  $\Delta(\Gamma(G, \ol p))$.  The following hold.

\n (i) If $\+ D$ is   $\log$-compressible,  then  so is $\+ C$. 

\n (ii)  If  $\+ D$ is strongly  $\log$-compressible,  then  so is $\+ C$.} \end{proposition}  
\begin{proof} Let $G\in \+ C$, so that $G = \Delta(F)$ for some $F  \in \+ D$. Let $\phi$ be a sentence of length $O(\log  |F|)$ describing $F$. The sentence $\psi$ expresses the following  about an $L$-structure $H$:

{\it  there are parameters $\ol q$ in $H$  such that $\Gamma(H, \ol q) \models \phi$ and 

\hfill $\eta$ describes an isomorphism $H \cong \Delta(\Gamma(H, \ol q))$.  }
 
We claim that $\psi$ describes $G$.  To see this, note that certainly $G \models  \psi$ via $\ol p$. If $ \wt G$ is an $L$-structure satisfying $\psi$ via   a list of parameters $\ol q$, then 
$\Gamma(\wt G, \ol q) \models \phi$ implies that $\Gamma(\wt G, \ol q) \cong F$, so that $\wt G \cong  \Delta(F) \cong G$.

To see   that $|\psi| = O(\log(|G|))$, recall that  the uniform interpretations   are by definition based on fixed sets of formulas.  Therefore  $|\psi| = O( |\phi|)$. Since  $\log|G|= O( \log |F|)$ by the remark after Definition~\ref{def::interpretation}, we have  $|\psi| = O(\log |G|)$. This shows (i).

To prove (ii) suppose that $G \in \+ C, G = \Delta(F)$ as above. Suppose  $g$ is a tuple in $G$; for notational simplicity assume its length is 1. Then $g$ is given by a $k$-tuple $u$ in $F$ for fixed $k$; we denote this by  $( G, g )= \Delta( F, u)$.  This tuple in  turn is given by a $k \cdot l$-tuple $w$ in $G$ when an appropriate  list $\ol q$ of parameters is fixed; we write $( F, u  ) = \Gamma(G , \ol q, w)$. 

Now by hypothesis on $\+ D$ there is a formula  $\theta (x_1, \ldots, x_k)$ of length $O(\log(|F|)$ describing $(F, u )$. Obtain a formula $\chi(y)$ by adding  to the expression for $\psi$ above   the condition on   $y$ that there is a $k \cdot l$ tuple $w$ of  elements of $H$ such that  $\Gamma(G , \ol q, w)$ satisfies $\theta$, and $\Delta( \Gamma(G , \ol q, w)) = (H, y )$. Then $|\chi| = O (\log |G|)$ and $\chi$ describes $( G, g )$.
\end{proof}
Note that if $\phi$ is a $\Sigma_k$ sentence, then $\psi$ is a $\Sigma_{k+c}$ sentence for  a constant  $c$ depending only on  the interpretations and the formula $\eta$. Thus, if  $\+ D$ is   $\log$-compressible using $\Sigma_k$ sentences,  then   $\+ C$ is $\log$-compressible using $\Sigma_{k+c}$ sentences.

The previous proposition allows us to deal with the class of   Ree groups of type $^2G_2$ using a result of Ryten. Note that  the class of difference fields  

\n $(\FF_{3^{2k+1}},\Fr_3^{k+1})$, $k \in \NN$,  is denoted   $\+C_{(1,2,3)}$ there. The  following is a special case of the more general result of Ryten.

\begin{theorem}\rm{(by \cite{Ryten:07}, Prop.\ 5.4.6(iii))}\label{t:ryten} Let $\+ C$ be the class of finite groups $^2G_2(q)$, $q = 3^{2k+1}$,  and let $\+ D$ be  the class of finite difference fields $(\FF_{3^{2k+1}},\Fr_3^{k+1})$. The hypotheses of Prop.~\ref {t:biinterpretation} can be satisfied via    uniform   interpretations   $\Delta, \Gamma$ and a formula $\eta$ in the language of groups. 
\end{theorem}
The details of the proof are contained in Ch.\ 5 of \cite{Ryten:07}. Since they
require quite a bit of background on  simple groups of Lie type, we merely  indicate
how to obtain the required formulas. The group $^2G_2(\FF)$ has Lie rank~$1$,
and hence behaves similarly to the group $SL_2(\FF)$, which also has Lie rank~$1$. The formulas required for Prop.~\ref {t:biinterpretation} are
essentially the same in both cases. Since most readers will be more 
familiar with $SL_2(\FF)$, we use this group  rather than $^2G_2(\FF)$ to make the required subgroups more explicit.

The uniform interpretation $\Delta$  of $\+C$ in $\+D$ is essentially
the same as in the case of the interpretation of $SL_2(\FF)$ in $\FF$ described above using
the fact that $G_2(\FF)$ - and hence its subgroup $^2G_2(\FF)$ - has a linear representation as a group of matrices. The groups $G_2(\FF)$ are uniformly definable in $\FF$ (as matrix groups which preserve the octonian algebra on $\FF$). The subgroups $^2G_2(\FF)$ of $G_2(\FF)$ are then uniformly defined in the language of difference fields by expressing that its elements induce linear transformations (of the affine group $G_2(\FF)$) that commute with the field automorphism $\sigma$.

The uniform interpretation with parameters $\Gamma$ of $\+D$ in $\+C$ can be given roughly
as follows: for the group $^2G_2(\FF)$, the \emph{torus} $T$ and the \emph{root subgroups} $U_+,U_-$ of $^2G_2(\FF)$ are uniformly definable subgroups (in the language of groups) using parameters from the group.

In the case of the group $SL_2(\FF)$, the torus is (conjugate to) the group $T$ of diagonal matrices in $SL_2(\FF)$ which can be defined uniformly as the centralizer of a nontrivial element $h$ in $T$. (The same holds for the group  $^2G_2(\FF)$.)

The root group $U_+$ of $SL_2(\FF)$ can be described as the upper triangular
matrices with $1$'s on the diagonal, similarly $U_-$ are the strict lower triangular matrices. The groups $U_+,U_-$ are isomorphic to the additive
group of the field $\FF$ (this is easy to see in the case of $SL_2(\FF)$).
  The torus $T$ acts by conjugation on $U_+,U_-$
as multiplication by the squares in $\FF$. As the characteristic of $\FF$ is $3$, any element of $\FF$ is the 
difference of two squares.  Thus the groups $U_+,U_-$ can be defined uniformly  by picking a nontrivial element $u$ in $U_+, U_-$,  respectively and considering the orbit $\{u^h\colon h\in T\}$ of $u$ under the conjugation by elements from $T$. Writing the group operation on $U_+,U_-$ additively, the set of differences $\{u^h-u^{h'}\colon h,h'\in T\}$ is uniformly definable and defines the root
groups. This also shows that from $U_+\rtimes T$ we definably obtain the field $\FF$. Again, for $^2G_2(\FF)$ this is essentially the same.

It remains to find a formula describing the isomorphism $\eta \colon H \cong \Delta(\Gamma(H, \ol q))$ for a group $H \in \+ C$ and an appropriate list of    parameters including the ones given above. For this we need the fact that by the 
Bruhat decomposition (see \cite{Carter:89}, Ch.\ 8, in particular 8.2.2)
we have  $^2G_2=BNB=B\cup BsB$ where in this case $B=U_+T$, $N$ is the normalizer
of $T$ and $s$ is (the lift of) an involution generating  the \emph{Weyl group} $N/T$ of $^2G_2$. 
Thus
any element of  $^2G_2$ (or in fact of any group of Lie type of Lie rank $1$) can be written 
uniquely either as a product of the form $u_1 h$ or of the form $u_1hsu_2$ where $u_1,u_2\in U_+, h\in T$ and $s$ is a fixed generator of the Weyl group of  $^2G_2$, i.e.\  $s\notin T$ normalizes $T$ and $s^2\in T$. This yields the required isomorphism~$\eta$.\qed

\medskip

\noindent  
\emph{Proof of  Proposition~\ref{p:Ree}}.
By Theorem~\ref{t:ryten} the class $\+C$ of  Ree  groups of  
type $^2G_2$ is uniformly parameter interpretable in the 
class $\+D$ of finite difference fields $(\FF_{3^{2k+1}},\Fr_3^{k+1})$. By Corollary~\ref{c:fieldtuple}, the class 
$\+D$ is strongly $\log$-compressible using $\Sigma_3$ sentences.  By Proposition~\ref{t:biinterpretation} (and the remark after its proof), this  implies that the class $\+C$ is strongly  $\log$-compressible via $\Sigma_d$ sentences for some constant $d$. (We estimate that $d \le 10$.) \qed

\begin{remark} In fact, Ryten proves that for fixed Lie type $\LL$ and rank $n$, the
class of finite simple groups $\LL_n$ is uniformly parameter bi-interpretable with
the corresponding class of finite fields or difference fields.  This  means that in addition  to the properties given
in Prop.~\ref {t:biinterpretation}
there is   a formula $\delta$ in the first-order language for $K$ that defines for each $F\in\+D$ an isomorphism between $F$ and $\Gamma(\Delta(F),\ol p)$.
 Via 
Proposition~\ref{t:biinterpretation} this yields a proof that
each class of finite simple groups is $\log$-compressible. However, since there
are infinitely many such classes, further effort would be needed  in order
to show that there is a single  $O$-constant which works for all classes. We have  circumvented the
problem by using the results of Guralnick et al.\ \cite{Guralnick:08}.
\end{remark}

%

 \begin{remark} By Remark~\ref{binary} and the proofs above,   each finite simple group $G$ actually has  a description of binary  length $O(\log(|G|))$.

  \end{remark}

 
%
Based on the methods above we can somewhat strengthen Theorem~\ref{t:simple}.
 \begin{prop} \label{prop:char simple} The class of characteristically simple finite  groups $G$ is $\log$-compressible. \end{prop}
 \begin{proof} Any  nontrivial characteristically simple finite  group  $G$ is  isomorphic to a direct power $S^k$, $k \ge 1$,  where $S$ is a   simple group (see e.g.\ Wilson~\cite[Lemma 2.8]{Wilson:09}). 
 Firstly  we consider the case that $S$ is abelian, and so cyclic of order $p$. The sentence describing $G$    expresses that there are $x_1, \ldots, x_k$ of order $p$ such that $x_1, \ldots, x_k$ generate the group (using the formulas $\alpha_k$ for $v= |G|$ from Lemma~\ref{generation}); we  can  say within length $O(\log |G|)$ that the $x_r$ commute pairwise  by  expressing with two disjunctions of length   $O(k)$ that  
 \[ \fa z_1 \in \{x_1, \ldots, x_k\} \fa z_2 \in \{x_1, \ldots, x_k\} \, [z_1, z_2 ] = 1. \]
 
 Now suppose that $S$ is nonabelian. It is well-known that $S$ can be generated by just two elements $g,h$.  In the following let $r$ range over $1, \ldots, k$.  The sentence $\phi$ describing $G$    starts with a block of existentially quantified variables $x_1, \ldots, x_k$ and $y_1, \ldots, y_k$; we think of $x_r, y_r$ as the generating set $g,h$ in  the $r$-th  copy of $S$. Firstly, we require,  using the $\alpha_{2k}$ for  the size $|G|$,   that  the set $\{x_1, \ldots, x_k, y_1 , \ldots, y_k\}$ generates the group $H$  under consideration,   that $[x_r, y_r]\neq 1$ for each $r$, and that 
 \[\fa z \in \{x_1, \ldots, x_k\} \ex^{ \le 1}  \, w \in \{y_1, \ldots, y_k\} \, [z,w ] \neq 1.\]
 This ensures that the subgroup  $U_r$ generated by $x_r, y_r$ is normal in the group $H$; hence so is its centraliser $C(x_r,y_r)$.  
  
 Secondly,  let $\phi_S$ be a description  of $S$ with $|\phi_S | = O(\log|S|)$ according to Theorem~\ref{t:simple}. We require that the center of $U_r$ is trivial (this is possible using the formula $\alpha_2$ for  size $|G|$), and that 
 $H /C(z, w) \models \phi_S$ for each $z \in \{x_1, \ldots, x_k\} $ and $w \in \{y_1, \ldots, y_k\}$ such that $[z,w] \neq 1$.  This can be done  within the required length bound since $C(z,w)$ is defined by a formula of fixed length. Since  $H \models \phi$ implies $U_r \cong H/C(U_r)$ for each $r$, the sentence describes $G$.
 \end{proof}

\section{Background on group extensions} \label{s:ext}
In this section we provide the tools needed for obtaining short first-order descriptions of general finite groups in Section~\ref{s: general}. 
To obtain such descriptions, we will use a  composition series of the group in question. Besides  describing the simple quotients,  
 we will also need to describe  the extension of  a group $N$ by a group $H$.
Such an extension can be understood via the second cohomology groups of certain associated modules.
Here  we give a more elementary account of the relevant  part of the theory of group extensions, an account which we can translate into a first-order description of the extension.
We consider a group extension $E$  
containing $N$  as a normal subgroup  such that $E/N\cong H$. (While all of this is in principle well-known,
we include it to keep the paper  self-contained in this regard.)

In contrast to the presentations of  Section 5, we will   use \emph{profinite} presentations for the group $H$ because in this setting it is known that a  small number of relators suffices. So we consider a  presentation
\[H\cong F/R\]
where $F=\widehat F(s_1,\ldots, s_k)$ is the profinite completion of the free group of rank $k$ on generators
$s_1,\ldots s_k$ and $R$ is the closed normal subgroup of $F$ topologically
generated (as a normal subgroup) by $r_1,\ldots, r_m$. For detail see e.g.\  Lubotzky and Segal~\cite[p.\ 47]{Lubotzky.Segal:03}.

We will show that any group extension $E$ of $N$ by $H$ is determined by the 
action of $F$ on $N$,  and an $F$-homomorphism from $R$ into $N$. Such a homomorphism is determined by the generators of $R$ as a normal subgroup, i.e., the relators for the profinite presentation, which is 
why we want the presentation to have as  few relators
as possible.

Let $E$ be an extension of $N$ by $H=\seq{s_1,\ldots, s_k}$.
Let $  s'_1,\ldots   s'_k\in E$ be \emph{lifts} of $s_1,\ldots, s_k\in H$, i.e. $\pi_H(  s'_i)=s_i, i=1,\ldots, k$.
Then the $  s'_i,i=1,\ldots k,$ act on $N$ by conjugation and hence any \emph{word} $w(\ol s)=w(s_1,\ldots, s_k)$ in the profinite
free group $F$ with generators $s_1,\ldots,s_k$ acts on $N$ (as an automorphism of $N$) via the natural
action of $w( s'_1,\ldots,  s'_k)\in E$. By  density this extends to a unique continuous   action of all of $F$ on $N$.
Hence any group extension $E$
of $N$ by a $k$-generated group $H=\langle s_1,\ldots, s_k\rangle$ determines  an action of $F=\widehat F(s_1,\ldots, s_k)$ on $N$, where the $s_i$ are now seen as generators of  $F$, rather than as elements of $H$. 
In order to describe $E$ we will have to express this  action of  $F$ on $N$.

Define $$\phi_E: R\longrightarrow N \ \text{by} \ w(s_1,\ldots,s_k)\mapsto w(  s'_1,\ldots,  s'_k)$$
and extend the definition  to the unique continuous function defined on all of $R$.
Then $\phi_E\in\Hom_{F}(R,N)$.
The next lemma states that the group $E$ is determined -- up to an  isomorphism over $N$ -- by the
action of $F$ on $N$ and the homomorphism $\phi_E$. 

\begin{lemma}\label{l:phi_E}
Using the previous notation, suppose that  $E^1, E^2$ are groups with a common normal subgroup $N$ and let $s^j_i\in E^j,j=1,2, i=1,\ldots, k$ be lifts of $s_1,\ldots, s_k$, respectively,
such that $(E^j/N,\ol s^j)\cong (H,\ol s), j=1,2$. 

Suppose that the induced $F$-actions agree, i.e. for all $a\in N$ we have
\[a^{s^1_i}=a^{s^2_i}, i=1.\ldots k. \tag{*}\]
Then  $E^1$ and $E^2$ are isomorphic over $N$ via an isomorphism
taking $ s_i^1$ to $s_i^2, i=1,\ldots k,$ if and only if  $\phi_{E^1}=\phi_{E^2}$ .

\end{lemma}
\begin{proof}
First suppose that $\phi_{E^1}=\phi_{E^2}$. Define for $a \in N$
\[f: E^1\longrightarrow E^2,\ aw(\ol s^1) \mapsto \ aw(\ol s^2).\]
Note that 
\[aw(\ol s^1) =a'w'(\ol s^1) \LR w(\ol s^1)(w'(\ol s^1))^{-1}\in N \LR w(\ol s)(w'(\ol s))^{-1}\in R.\]
 Since $\phi_{E^1}=\phi_{E^2}$, we
see that indeed $f$ is well-defined. Exchanging the roles of $E^1$ and $E^2$ shows that $f$ is injective.

Note that $f$ is a homomorphism because the $F$-actions on $N$ agree:
let $a_0, a_1 \in N$, and let $w_0, w_1$ be group words in  variables $\ol s=s_1, \ldots, s_k$. Then
\begin{eqnarray*} f(a_0w_0(\ol s^1) a_1 w_1(\ol s^1) )& = & f(a_0 a_1^{w_0^{-1}(\ol s^1)} w_0(\ol s^1) w_1(\ol s^1))  \\
&= & a_0  a_1^{w_0^{-1}(\ol s^2)} w_0(\ol s^2) w_1(\ol s^2)\mbox{\hspace{1.5cm} by }(*)  \\
&= & a_0 w_0(\ol s^2) a_1 w_1(\ol s^2) \\
  & = & f(a_0w_0(\ol s^1) ) f  (a_1 w_1(\ol s^1) )\end{eqnarray*}

Since $E^j$ is generated by $N$ and $\ol s^j,j=1,2$, this now implies that $f$ is surjective
and hence an isomorphism fixing $N$ pointwise.

For the converse implication, suppose that $g:E^1\longrightarrow E^2$ is an isomorphism fixing $N$ pointwise and taking $s_i^1$ to $s_i^2, i=1,\ldots k,$. For any word $w$ with $w(\ol s)\in R$ we   have \bc $g(w(\ol s^1 ))=w(\ol s^1 )=\phi_{E^1}(w(\ol s))$. \ec 

Also
\bc $g(w(\ol s^1))=w(\ol s^2 )=\phi_{E^2}(w(\ol s))$, \ec proving the lemma.
\end{proof}
A close inspection of the proof of Lemma \ref{l:phi_E} yields the following
variant, which will be used in   Section~\ref{s: general}  for the first-order description of group
  extensions.

\begin{lemma}\label{l:3m}
Suppose that in the situation of Lemma \ref{l:phi_E} every element of $H$ has length at most $m$ with respect to $\ol s $.
Then $E^1$ and $E^2$ are 
isomorphic over $N$ provided that  $\phi_{E^1}^{3m}=\phi_{E^2}^{3m}$, where
$\phi_{E^j}^{3m}, j=1,2,$ denotes the restriction of $\phi_{E^j}$ to the elements of $R$ of word length 
at most $3m$ over $\ol s$.
\end{lemma}
\begin{proof}
Define $f: E^1\longrightarrow E^2$ by
\[a w( \ol s^1)\mapsto a w( \ol s^2),\]
where $ a\in N$ and $w \in F(\ol s )$  is a group word such that $ |w|\leq m$. 
By assumption, $f$ is defined on all of $E^1$. We verify  as in the proof of Lemma~\ref{l:phi_E}   that $f$ is well-defined  and injective, noting that only words of length $\le 2m$ are relevant now.

To check that $f $ is a homomorphism, let $a_0, a_1 \in N$, and let $w_0, w_1$ be group words in  variables $s_1, \ldots, s_k$ of length at most $m$. 
By assumption
there are $a\in N$ and a word $w_2$ of length
at most $m$ such that

\[  w_0(\ol s^1 )w_1(\ol s^1)=aw_2(\ol s^1 ).\]

Since $\phi_{E^1}^{3m}=\phi_{E^2}^{3m}$ we have

\[ w_0(\ol s^2)w_1(\ol s^2 )=aw_2(\ol s^2 ).\]
Hence   as in the proof of Lemma~\ref{l:phi_E} we have
\begin{eqnarray*} f(a_0w_0(\ol s^1) a_1 w_1(\ol s^1) )& = & f(a_0 a_1^{w_0^{-1}(\ol s^1 )} w_0(\ol s^1) w_1(\ol s^1))  \\
& = &f(a_0 a_1^{w_0^{-1}(\ol s^1 )}a w_2(\ol s^1)) \\
&= & a_0  a_1^{w_0^{-1}(\ol s^2)} a w_2(\ol s^2) \mbox{\hspace{1.5cm}  by (*)  }\\
&= & a_0 w_0(\ol s^2) a_1 w_1(\ol s^2 )  \\
  & = & f(a_0w_0(\ol s^1) ) f  (a_1 w_1(\ol s^1 ) )\end{eqnarray*}
Since $E^j,j=1,2$, is generated by $N$ and $\ol s^j$, this now implies that $f$ is surjective
and hence an isomorphism fixing $N$ pointwise.
\end{proof}

Recall that a group action is called \emph{regular}
if it is transitive and point stabilizers are trivial.

\begin{lemma}\label{l:homology_group}
Let $Z=Z(N)$.
The group $\Hom_F(R,Z)$ acts regularly on the set \[X=\{\phi_E\colon E\mbox{ is extension of } N \mbox{ by } H \mbox{
with prescribed }  F\mbox{-action on } N\}\]
via $\phi_E^\psi(w(\ol s ))=\phi_E(w(\ol s ))\psi(w(\ol s ))$ for $\psi\in Hom_F(R,Z)$ and $\phi_E\in X$
\end{lemma}
\begin{proof}
To see that the action is transitive just notice that for extensions $E_1, E_2$ of $N$ by $H$ with the given $F$-action on $N$, and lifts $s_i^j, j=1,2, i=1,\ldots k$ as before
we have  for all $n\in N$
\[n^{\phi_{E_1}(w(\ol s))}=n^{w(\ol s^1)}=n^{w(\ol s)}=n^{w(\ol s^2)}=n^{\phi_{E_2}(w(\ol s))}\]
and hence $\phi_{E_1}(w(\ol s ))(\phi_{E_2}(w(\ol s )))^{-1}\in Z$. By
continuity,  $\phi_{E_1}$ and $\phi_{E_2}$ differ
by an element in $\Hom_F(R,Z)$.

Let $\psi \in Hom_F(R,Z)$.  To see that $\phi_E^\psi=\phi_{E^1}$ for some extension $E^1$ with prescribed $F$-action on $N$,
define $E^1$ by choosing a transversal $T$ for $F/R$ so that any element $w(\ol s )\in F$ can be written uniquely
as \bc  $w(\ol s)=v(\ol s )r(\ol s )$ \ec where $v(\ol s )\in T,r(\ol s )\in R$.

Let $ s^0_i,i=1,\ldots k$ be the lifts of $s_i$ to $E$.
We now define an extension $E^1$ with lifts $s^1_i,i=1,\ldots k$, by letting the elements of $E^1$ be
\[ nw(\ol s^1)=nv(\ol s^0)\phi_E(r(\ol s ))\psi(r(\ol s ))\]
with the induced multiplication. Then $E^1$ is an extension with prescribed $F=\widehat F(\ol s )$ action and $\phi_{E^1}=\phi_E^\psi$.
\end{proof}

The  \emph{rank} of an abelian group  $A$, denoted $\rank A$, is the minimal size of a   set of generators, or, in other words, the least $k$ such that there is an onto map $\ZZ^k \to A$. Clearly $B \le A$ implies $\rank B \le \rank A$. 
Letting $\lambda n$ denote the number of prime factors of $n$ with multiplicity,
we have $\rank A\leq \lambda |A|\leq\log|A|$.

\begin{remark}\label{rem:dimension}
Lemma \ref{l:homology_group} implies that the number of extensions of $N$ by $H$
is at most $|Z|^r$ where $Z=Z(N)$ and $r$ is the minimum number  of generators
of $R$ as a closed normal subgroup of $F$.
The \emph{rank} of $\Hom_F(R,Z)$   is at most $r\cdot \lambda|Z|$ since each $\phi \in \Hom_F(R,Z)$ is determined by its values on the $r$  generators of $R$.
\end{remark}

\begin{corollary}\label{c:trivialrad}
 If $Z(N)=1$, then an extension $E$ of $N$ by $H$ is determined up to isomorphism
over $N$ by the $F$-action on $N$. 
\end{corollary}

Lemma~\ref{l:3m} states that the restriction $\phi_E^{3m}$ of $\phi_E$ to words of length at most $3m$ is sufficient for describing an extension $E$.
To give a short description of $\phi_E^{3m}$, we heavily rely on the following  lemma originally suggested by Alex Lubotzky.

\begin{lemma}\label{l:linalg}
Let $A$ be a finite abelian group, $X$ a set and let $V\leq A^X$ be a subgroup of  rank $d$. There exists a set $Y \sub  X$ of size at most $ d\cdot\lambda(|A|)$ such that for all $g \in V$,   $g\upharpoonright Y = 0$   implies  $g= 0$. 

\end{lemma}

\begin{proof}
Decompose $A$ into its $p$-primary components $A=\bigoplus_pA_p$. Since the number of different primes dividing the order of $A$ is at most $\lambda(|A|)$, the lemma follows from applying Lemma~\ref{l:linalgp} below to each of the $A_p$
separately.
\end{proof}
\begin{lemma}\label{l:linalgp}
Let $A$ be a finite abelian \emph{$p$-group}, $X$ a set and let $V\leq A^X$ be a subgroup of  rank $d$. There exists a set $Y \sub  X$ of size at most $ d$ such that for all $g \in V$,   $g\upharpoonright Y = 0$   implies  $g= 0$. 
\end{lemma}
\begin{proof}
Since $A$ is a direct product  of $k$ cyclic $p$-groups for some $k$, we may consider $V\leq A^X\leq (C^k_q)^X\cong (C_q)^{k|X|}$ where $q$  is the exponent of   
$A$. 
Then, replacing  each element $x$ of $X$ by   $k$ new elements $\la x,1\ra, \ldots, \la x,k\ra$, we may assume $A=C_q$ and $V\leq C_q^{X'}$ where  $X' = X \times \{1, \ldots, k\}$. Once we have  found a  subset $Y'$ of $X'$ of size at most $d$ with the required property,  we obtain $Y\subseteq X,|Y|\leq d,$ by replacing each  element $\la y,i \ra \in Y'$ by  $y$. If $g \in V$ and $g(y)=0$ then $g(\la y,i\ra)=0$ for each $i$  when $g$ is viewed as a function on $X'$ with values in~$C_q$.

Without loss of generality we may thus assume that $A =C_q$. For $x\in X$ let $g_x\colon V\to A$ denote the  coordinate function mapping $p \in V$ to $p(x)$. There is nothing to show if
$d=0$, so suppose $d>0$.

Let $x_1\in X,  v_1\in V$ such that $g_{x_1}(v_1)$ has maximal order. Then $g_{x_1}(V)\leq \ZZ g_{x_1}(v_1)$.  
We claim that  $V$ decomposes
as $V= \ZZ v_1\oplus \ker(g_{x_1})$. First note that clearly $\ZZ v_1 \cap \ker (g_{x_1}) = 0$. Next, given arbitrary  $w \in V$, choose $r \in \ZZ$ so that $g_{x_1}(w)= r \cdot g_{x_1}(v)$. Then $g_{x_1}(w-r\cdot v) =0$, so that $ w \in  \ZZ v_1+\ker(g_{x_1})$.


  Clearly  $\rank\ker(g_{x_1})\leq d-1$ and we may consider $\ker(g_{x_1})$ as a subgroup of $A^{X\setminus\{x_1\}}$. 
Inductively we find $x_2,\ldots, x_d\in X\setminus \{x_1\}$
 with corresponding elements $v_2,\ldots, v_d$ such that $V=\bigoplus_{i\leq d}\ZZ v_i$ and
$\bigcap_{i\leq d}\ker(g_{x_i})=0$. Hence
 $Y=\{x_1,\ldots, x_d\}$ is as required.     
\end{proof}

\begin{remark} The bound given in Lemma~\ref{l:linalg} is optimal. For suppose $A$ is a product of $n$ cyclic groups of different prime orders $p_1, \ldots, p_n$ with generators $g_1, \ldots, g_n$. Let $X= \{x_1, \ldots, x_n\}$, and let $f(x_i) = g_i$. For each $i$, $\la f \ra$ contains an element that only differs from $0$ at the  $i$-th component. \end{remark}

We   summarize and assemble all the pieces of this section in the following
proposition, which we will use in the next section to  describe arbitrary finite groups. 
  
\begin{proposition} \label{prop:determine_E}
Suppose  $H=F/R$  where $F=\widehat F(s_1,\ldots, s_k)$ and $R$ is generated as a closed normal subgroup of $F$ by $r$ elements.
Let $\ol s\subset H$ be the image of $(s_1, \ldots, s_k)$, so $\seq{\ol s}=H$,
and suppose that any element of $H$ has length at most $m$ over $\ol s$.
Let $N$ be a finite group, and let  $Z=Z(N)$.

There are words 
\bc $w_1,\ldots, w_d\in R$ of length at most~$3m$, where $d=r\cdot\lambda|Z|$,\ec
 such that group extensions
$E^j, =1,2,$ of $N$ by $H$ are isomorphic over $N$
under an isomorphism taking  a lift  $\ol s^1\in E^1$  of $\ol s$
to  a  lift $\ol s^2\in E^2$,  provided the following conditions hold:
\begin{enumerate}
\item[(a)] $a^{s_i^1}=a^{s_i^2}, i=1,\ldots, k$ for all $a\in N$;
\item[(b)] $(E^1/N,\ol s^1)\cong(E^2/N,\ol s^2)$;
\item[(c)] $w_i(\ol s^1)=w_i(\ol s^2)\in N$ for $ i=1,\ldots, d$.
\end{enumerate}
\end{proposition} 
\begin{proof}
 By   Lemmas \ref{l:3m} and \ref{l:homology_group}, the abelian group $\Hom_F(R,Z)$  can be seen as a subgroup of $Z^X$ where $X$ is the set of group words in $s_1, \ldots, s_k$ of length $\le 3m$. 
Note that we have $\Hom_F(R,Z)=\bigoplus_p\Hom_F(R,Z_p)$ where $Z_p$ are the $p$-primary components of $Z$. Now for each prime $p$, the group  $\Hom_F(R,Z_p) $  can be seen as a subgroup of $Z_p^X$, and
by   Remark~\ref{rem:dimension} $\rank \Hom_F(R,Z_p)\leq  r\cdot \lambda |Z_p|$ .
Since $\sum_p r\cdot \lambda |Z_p|=r\cdot \lambda|Z|=d$,
we can use Lemma \ref{l:linalg} to find the required $w_1, \ldots w_d\in X$.
\end{proof}

We will apply the previous proposition in the situation where $H$ is a finite simple group, $\{s_1,\ldots, s_k\}\sub H$ is a swift
generating set of $H$ of size at most $\log|H|$ and $H$ has a profinite presentation $H\cong F/R$
where $R$ is generated as a closed normal subgroup of $F$ by $O(\log|H|)$
elements. The existence of such a profinite presentation is guaranteed by results in \cite{Lubotzky.Segal:03} and \cite{Guralnick:08}:

\begin{theorem}\label{t:profinite}\cite{Lubotzky.Segal:03}
There is a constant $C$ such that
any finite simple group  generated by $d$ elements has a profinite presentation with $d$  generators and $C+d$ relations.
 \end{theorem}
\begin{proof}
This follows from Theorem B of \cite{Guralnick:08} and the proof of Theorem 2.3.3 of \cite{Lubotzky.Segal:03} in the case of simple groups. For the latter, we 
refer to the proof that Conjecture B implies Conjecture A in \cite{Lubotzky.Segal:03}. This proof also  works for profinite presentations as verified  by the authors after Thm.\ 3.3. 
\end{proof}

\section{Describing general finite groups} 
\label{s: general}

We are now in the position to give short descriptions of arbitrary finite groups.

\n {\bf Theorem \ref{t:main}.}\emph{ The class of finite groups is strongly  $\log^3$-compressible.  }

\begin{proof}  Let $G$ be a finite group.
We fix a subnormal series 
 \[1=G_0\lhd G_1\lhd\ldots \lhd G_r=G\]
with simple factors
$H_i:=G_i/G_{i-1}$, $ i=1,\ldots r$. 

Note that the  length $r$ is  bounded  by $\log |G|$.

Choose   an ascending sequence of   sets 
\[ \ES = T_0\subset T_1\subset\ldots\subset T_r=T\]
with $\seq{T_i}=G_i,0 \le i\leq r\leq \log |G|$, 
as follows.


\vsps

\n (1)  If $H_i$ is a finite simple group not of type $^2G_2$, we let $T_i=T_{i-1}\cup \{s_1,\ldots s_C\}$ for any elements $s_1,\ldots, s_C\in G_i$ such that
$s_1 G_{i-1}, \ldots, s_C G_{i-1}$ are generators for $G_i/G_{i-1}= H_i$
according to Theorem \ref{thm:Guri}. Then $(H_i, s_1 G_{i-1}, \ldots, s_C G_{i-1})$ can
be described by  a sentence $\phi_i$ of length  $O(\log |H_i|)$ by Proposition~\ref{p:simple}(i) and its proof.

\vsps

\n (2)  If $H_i$ is a group of type $^2G_2$, we let  $C=2$ and $T_i=T_{i-1}\cup\{s_1,s_C\}$ for any
elements $s_1,s_2\in G_i$ such that $s_1G_{i-1},s_2G_{i-1}$ generate $G_i/G_{i-1}\cong H_i$. Then $(H_i, s_1G_{i-1}, s_2G_{i-1})$ can
be described by a sentence $\phi_i$ of length  $O(\log |H_i|)$ by Proposition~\ref{p:Ree}. 
 
\vsps

Note that for each $i=1,\ldots,r$ we have $|T_i\setminus T_{i-1}|\leq C_0$ for
the constant $C_0$ given in Theorem~\ref{thm:Guri}.
The sentence  $\phi$ describing $G$ needs to express conditions (a), (b) and (c) of Proposition~\ref{prop:determine_E} for each group $G_i$
with normal subgroup $G_{i-1}$. 

We start with a prenex of existential quantifiers referring to the elements of $T$.

\medskip

\item 1. Obtain a preprocessing set $A$ for $G$ over $T$:   

To give short descriptions for the conditions of Proposition~\ref{prop:determine_E}, 
at each step $i=1,\ldots r$ we first obtain preprocessing sets from $T$
using formula $\psi$ from Lemma~\ref{l:preproc_formula} where we replace
the parameters from $T$ by the corresponding variables. 
  The formula  $\psi$
has  length $O(\log^2|G|)$. 

Since the preprocessing set $A$ will be used in each part of the sentence $\phi$, the scope of the existential quantifiers referring to $A$   extends over all of~$\phi$. 
\medskip

\item 2. Express $(G_i/G_{i-1}, T_i\setminus T_{i-1})\cong (H_i,s_1 G_{i-1}, \ldots, s_C G_{i-1})$:

We let  the formula $\chi_i, \, i=1,\ldots r,$ express that
 \[      \ \ \   ( G_i/G_{i-1}, T_i\setminus T_{i-1}) \models \phi_i.   \]

 We can use the $\alpha_{|T_i|}$
to express that $G_{i-1}$ is a normal subgroup of $G_i$  using a length of $O(\log |G_i|)$.  
We now restrict the quantifiers in $\phi_i$ to $G_i$ using $\alpha_{|T_i|}$ and replace each occurrence of ``$u= v$''  in $\phi_i$ by  \bc ``$uv^{-1} \in G_{i-1}$''. \ec 
Since we replace the equality   symbols  in $\phi_i $ by  strings  of length  $O( \log |G_{i-1}|)$, the resulting formula $\chi_i$ has length $ O( \log|H_i|\log |G_{i-1}|)$. 
Then the  conjunction $\chi$ of the formulas $\chi_i$ has length $O(\log^2|G|)$.

\medskip

\item 3. Conjugation action of $G_i$ on $G_{i-1}$:

For each $i=2,\ldots, r$, let $\kappa_i$ describe the action of 
$g\in T_{i} \setminus T_{i-1}$ on $G_{i-1}$ by conjugation. 
Since $T_{i-1}$ generates $G_{i-1}$, it suffices  to determine  $g^{-1}w g$  
for each  $w \in T_{i-1}$ and $g \in T_{i}\setminus T_{i-1}$ as  an element $h_{w,g}\in G_{i-1}$.
Since $h_{w,g}$ has length at most $2\log|G_{i-1}|$ over $A_{i-1}$ and there are at most $C_0\cdot\log|G_{i-1}|$ such pairs, $\kappa_i$
  has length in $O(\log^2|G_{i-1}|)$. The conjunction $\kappa$ of the $\kappa_i$ has
length $O(\log^3|G|)$.

\medskip

\item 4. Describing  the extension of $G_{i-1}$ by $H_i$:

We use Theorem~\ref{t:profinite} to obtain a profinite presentation for $H_i$
with a swift generating set (Cor.\ \ref{c:swift_generation}) of size $k\leq \log|H_i|$ corresponding to the elements of $A_i\setminus A_{i-1}$ 
and with $r\leq C+\log|H_i|$ relations.

By Proposition~\ref{prop:determine_E} there is  $d\leq \log|Z(G_{i-1})|(C+\log |H_i|)$,  and there are 
words $w_1,\ldots w_d$ in $\ol a_i=A_i\setminus A_{i-1}$ of length at most $3\log |H_i| $ such that $w_j(\ol a_i)=h_j~\in~G_{i-1}, j=1,\ldots, d,$ determine  $ G_i$. 
Since any element of $G_{i-1}$ has length at most $2 \log|G_{i-1}|$ over $A_{i-1}$, we obtain 
a formula $\rho_i$ of length 
$O(\log|Z(G_{i-1})|\log|H_i|\log|G|)$. Since $\sum_i \log |H_i|\leq \log|G|$, 
the conjunction $\rho$ of the $\rho_i$ yields a formula of length 
$O(\log|G|\log^2|G|)$.

\medskip

We now let $\phi$ be the sentence consisting of the prenex of existential quantifiers referring to $T$ followed by the conjunction of $\psi,\kappa,\chi,$ and $\rho$. By repeated application of Proposition~\ref{prop:determine_E}
one verifies  that $\phi$ describes $G$. The strong $\log ^3$-compressibility of the class of finite groups follows
 since any element of $G$ has length at most $2\log|G|$ over the preprocessing set $A$,  \end{proof}

The strong   $\log^3$-compressibility of the class of finite groups allows
us to also describe finite transitive permutation groups (as explained in the introduction), and finite
groups with a distinguished automorphism.

\begin{cor}  {\rm (i)}  The class of finite groups with a distinguished subgroup is $\log^3$-compressible in the language of groups with an additional unary predicate.  \\
{\rm (ii)}  The class of finite groups with a distinguished automorphism is $\log^3$-compressible in the language of groups with an additional  unary function.  \end{cor} 
\begin{proof} (i) Given a finite group $G$ and a subgroup $U \le G$, choose a string $\ol g$  of generators for $U$ of length  $k \le \log |G|$. Let $\phi$ be the description of $( G, \ol g )$ obtained above. Then $|\phi| = O (\log^3 |G|)$. Use the formula $\alpha_k$ from Lemma~\ref{generation} of length $O( \log |G|)$ to express that $U = \seq {\ol g}$ in $G$.  
(ii) is similar. 
\end{proof}

\begin{remark} \label{rmk:Ottimal} The exponent $3$ in Theorem~\ref{t:main} is optimal even for $p$-groups of nilpotency class 2 by a result of  Higman, which states that    there are at least $p^{\frac 2 {27}  n^2(n-6)}$ non-isomorphic such  groups of order $p^n$ (see e.g.\ \cite[Thm. 4.5]{Blackburn}). This result is  applied in a way  similar to the proof of \cite[Prop.\ 8.6]{Babai:97}. \end{remark}

We provide an upper bound on  the length  of descriptions when only a  bounded number of quantifier alternations is allowed.
 
 \begin{theorem} \label{thm:bounded QA} For some $m$, the class of finite groups $G$  is   $\log^4$-compressible via $\Sigma_m$ sentences. \end{theorem} 
 \begin{proof} We only note the necessary modifications to the previous arguments.  Throughout, instead of the $\alpha_k$ we use the existential generation formulas $\beta_k$ from Lemma~\ref{generation2}, which have  length $O(\log^2|G|)$ because $k \le \log |G|$ throughout. 
 
 The new version of $\psi$ in Lemma~\ref{l:preproc_formula} now has length $O(\log^3|G|)$. For some small $d$ we can choose $\Sigma_d$-descriptions $\phi_i$ of  $H_i$ of  length $O (\log^2 |H_i|)$ via Propositions~\ref{p:simple} and~\ref{p:Ree}.  Since we now replace the equality   symbols  in $\phi_i $ by  strings  of length  $O( \log^2 |G_{i-1}|)$, the resulting new version of the  formula $\chi_i$ has length $ O( \log|H_i|\log^2 |G_{i-1}|)$, and their conjunction has length $ O (\log^3 |G|)$. No generation formulas are used elsewhere in the proof   of Theorem~\ref{t:main}, so we conclude the argument as before. It is clear that the number of quantifer alternations is now  bounded.\end{proof}

 \begin{remark}  \
 
 \n 1. We don't know  whether the exponent can be improved to $3$ in Thm.\ \ref{thm:bounded QA}. 

\n  2. Reviewing the proof of Theorem~\ref{t:main}, it would   be interesting to show  a stronger   compressibility result for  the class of finite groups without  nontrivial abelian normal subgroup. In this case, we have  $Z(G_i) =1$ for each $i$, so that  Step 4 is not needed.  \end{remark}

 \begin{remark}  Theorem \ref{t:simple} leaves open some questions. 
 It would be interesting to 
show $\log$-compressibility for    classes of finite groups $G$  that are in some sense  close to simple, similar to Prop~\ref{prop:char simple}. These include    the central extensions  of simple groups, 
and    the   almost simple groups (that is, $S \le G \le \text{Aut}(S)$ for some simple group $S$). 
Also  it would be desirable  to show the strong $\log$-compressibility for  the
class of finite simple groups. \end{remark}

\noindent
{\bf Acknowledgement:} The authors wish to thank Alex Lubotzky, Bill Kantor, and Martin Ziegler for helpful discussions. This research was partially supported by the Marsden fund of New Zealand,  a Hood fellowship by the Lions foundation of New Zealand, and the DFG through SFB 878.

 \bibliographystyle{plain}
%

\end{document}